\documentclass[a4paper, english, 10pt]{amsart}

\usepackage{amsfonts}
\usepackage{caption}
\usepackage{forest}
\usepackage{amsmath}
\usepackage{amssymb}
\usepackage{tikz}
\usepackage[cp850]{inputenc}
\usepackage[bookmarksnumbered,plainpages
]{hyperref}
\usepackage{color}
\usepackage{mathtools}
\usepackage{MnSymbol}
\usepackage{subfigure}
\usepackage[all]{xy}
\usepackage{amsaddr}

\allowdisplaybreaks
\newtheorem{theorem}{Theorem}[section]

\newtheorem{proposition}[theorem]{Proposition}
\newtheorem{corollary}[theorem]{Corollary}

\newtheorem{remark}[theorem]{Remark}
\newtheorem{lemma}[theorem]{Lemma}
\newtheorem{conjecture}[theorem]{Conjecture}

\newtheorem{example}[theorem]{Example}

\def\sym#1{\mathrm{Sym}(#1)}
\def\c#1{\mathrm{con}_{#1}}
\def\comment#1{{\color{red} #1}}

\def\aut#1{\mathrm{Aut}(#1)}

\def\aff#1{\mathrm{Aff}#1}

\def\lmlt{\mathrm{LMlt}}
\def\N{Norm}
\def\ldiv{\backslash}

\def\dis{\mathrm{Dis}}

\def\setof#1#2{\{#1\, : \,#2\}}

\newcommand*\xbar[1]{%
   \hbox{%
     \vbox{%
       \hrule height 0.5pt 
       \kern0.5ex
       \hbox{%
         \kern-0.1em
         \ensuremath{#1}%
         \kern-0.1em
       }%
     }%
   }%
}

\title{A conjecture on superconnected quandles}

\author{Marco Bonatto}

\address[Marco Bonatto]{}

\email{marco.bonatto.87@gmail.com}
\begin{document}

\begin{abstract}
We study simple superfaithful and superconnected quandles and we found counterexamples to a conjecture suggested by computational data. We provide also examples of superconnected quandles built using group theoretical results and investigate primitive quandles.
\end{abstract}

\maketitle

\section{Introduction}

Racks and quandles are binary algebraic structures that have been studied because of their connections with knot theory and statistical mechanics.  Such structures have been introduced independently by Joyce and Matveev \cite{J,Matveev} as algebraic invariants of knots and links.  

From a purely algebraic viewpoint, quandles have been studied using several tools and tecniques coming from group theory, module theory and universal algebra \cite{AG, hsv, Medial,CP}.  In particular, commutator theory (and the related notions of abelianness and solvability) in the sense of \cite{comm} and non-trivial Mal'cev conditions have been investigated respectively in \cite{CP} and \cite{Maltsev_paper}.

For a finite quandle, having a Mal'cev term is equivalent to the property of being {\it superconnected}. In this paper we want to investigate the building blocks of Mal'cev varieties of quandles, namely finite superconnected simple quandles.  An example of a class of superconnected quandles is given by finite left distributive quasigroups or {\it latin} quandles. According to the main results of \cite{Stein2} and of \cite{CP}, finite latin quandles are solvable in the sense of \cite{comm} and accordingly the simple ones are abelian \cite{Principal}. We wonder if this property holds also for finite superconnected quandles and so we formulate the following conjectures, also suggested by computational data.

%
\begin{conjecture}\label{conj1} \text{}
\begin{itemize}
	\item[(i)] Finite superconnected quandles are solvable.
	\item[(ii)] Finite simple superconnected quandles are abelian.
\end{itemize}
\end{conjecture}

Clearly Conjecture \ref{conj1}(i) implies \ref{conj1}(ii) and Conjecture(i) is true for {\it involutory} quandles, indeed superconnected involutory quandles are latin and then solvable \cite[Theorem 3.6]{Super}. Conjecture \ref{conj1}(i) has been verified computationally on the whole \cite{RIG} library of GAP and Conjecture \ref{conj1}(ii) is verified on an extended library of simple quandles up to size $10^3$ computed by Milan Cvr\v{c}ek.

Finite simple quandles are constructed over finite simple groups.  In the paper study superfaithful and superconnnected simple quandles and we are able to characterize such properties for simple quandles built over some particular simple groups. The main results in this direction is to show counterexamples to the Conjecture \ref{conj1}(ii) and in particular we show that such a conjecture is not true even when restricted to {\it primitive} quandles.

The paper is organized as follows. In Section \ref{Pre} we show all the basics and the preliminary results on left quasigroups and quandles.  In Section \ref{Sconj} we fucus on conjugation quandles and in particular on the property of being superfaithful and superconnected.  The main results on simple quandles are in Section \ref{simple} and Section \ref{primitive}.

\subsection*{Acknowledgements}
We would like to acknowledge Geoffrey Robinson for pointing us out the relevant results Theorem \ref{pi'2} and Theorem \ref{O_p'} and Keith Kearnes for suggesting the proof of Proposition \ref{star prop and complement}. The proof of Lemma \ref{lemma subgrouop containing diagonal} was sketched by Sean Eberhard. We also thank Daniele Toller for his help in cleaning up some of arguments and proofs.

\section{Preliminary results}\label{Pre}

\subsection{Left quasigroups and quandles}

A left quasigroup is an algebraic structure given by a set $Q$ endowed with two binary operations $*,\ldiv$ such that 
\begin{eqnarray}
x*(x\ldiv y)\approx y\approx x\ldiv (x*y). \label{lQG}
\end{eqnarray}
We usually denote such algebraic structure as $(Q,*,\ldiv)$ or simply by $Q$. Note that according to \eqref{lQG} the mapping $L_x:y\mapsto x*y$ is a permutation for every $x\in Q$ (indeed $L_x^{-1}:y\mapsto x\ldiv y$). Hence we can define the {\it left multiplication group} of $Q$ as $\lmlt(Q)=\langle L_x,\, x\in Q\rangle$.  A left quasigroup $Q$ is said to be {\it connected} if $\lmlt(Q)$ is transitive over $Q$.  If every subalgebra of $Q$ is connected,  we say that $Q$ is {\it superconnected}. We refer to the orbit of $x\in Q$ with respect to the action of $\lmlt(Q)$ just as the {\it orbit of $x$}. It is easy to verify that the orbits of $Q$ are subalgebras.

The equivalence relation $\lambda_Q$ defined by setting $x\,\lambda_Q\, y$ if and only if $L_x=L_y$ is called the {\it Cayley kernel} of $Q$. A left quasigroup $Q$ is said to be {\it faithful} if $\lambda_Q$ is the identity relation. If every subalgebra of $Q$ is faithful then $Q$ is said to be {\it superfaithful}.  If the mapping $R_x:y\mapsto y*x$ is bijective for every $x\in Q$ then $Q$ is called {\it latin}. It is easy to see that idempotent superconnected and latin left quasigroups are superfaithful and finite latin left quasigroups are superconnected \cite{Super}.

A bijection $f$ on the set $Q$ is an automorphism of $(Q,*,\ldiv)$ if $f(x*y)=f(x)*f(y)$ for every $x,y\in Q$ (if so then also $f(x\ldiv y)=f(x)\ldiv f(y)$ for every $x,y\in Q$).  We denote the automorphism group of $Q$ as $\aut{Q}$.  

The left quasigroup $(Q,*,\ldiv)$ is {\it idempotent} if it satisfies the identity
\begin{eqnarray}
x*x\approx x.\label{ID}
\end{eqnarray}
A {\it rack} is a left quasigroup $(Q,*,\ldiv)$ such that the identity
\begin{eqnarray}
x*(y*z)\approx (x*y)*(x*z)\label{LD}
\end{eqnarray}
holds. According to \eqref{LD} the left multiplication group of a rack is a subgroup of $\aut{Q}$.  Idempotent racks are called {\it quandles}. 

Let us introduce some constructions of quandles over groups. 

\begin{itemize}
\item Let $G$ be a group and $H$ a subset of $G$ closed under conjugation (e.g. a union of conjugacy classes in $G$).  We denote by $Conj(H)$ the quandle with underlying set $H$ and the operation $x*y=xyx^{-1}$ for every $x,y\in H$.

\item Let $G$ be a group,  $\theta \in \aut{G}$ and $H\leq Fix(\theta)=\setof{x\in G}{\theta(x)=x}$.  The set of cosets $G/H$ endowed with the operation $xH*yH=x\theta(x^{-1} y)H$ for every $x,y\in G$ is a quandle.  We denote such quandle by $\mathcal{Q}(G,H,\theta)$ and we refer to it as {\it coset quandle}.  The group $G$ is a transitive group of automorphism of $\mathcal{Q}(G,H,\theta)$ (acting by $g\cdot xH=gxH$ for every $x,g\in G$).  On the other hand,  it is a very well know fact that if $Q$ is a quandle and $G\leq \aut{Q}$ is transitive over $Q$ then $Q$ is isomorphic to a coset quandle over $G$.  In particular connected quandles are basically coset quandles over their left multiplication groups.

\item Let $A$ be an abelian group and $f\in \aut{A}$. Then $Q=(A,*)$ where $x*y=(1-f)(x)+f(y)$ is a quandle denoted by $\aff(A,f)$. We say that $Q$ is an {\it affine} quandle over $A$.
\end{itemize}

Let us show how the first two constructions above are related.  

\begin{lemma}\label{quandles as conj}
Let $G$ be a group, $\theta\in\aut{G}$ and $H=G\rtimes \langle\theta\rangle$. Then
$$(1,\theta)^{H}=\setof{(g\theta(g)^{-1},\theta)}{g\in G}.$$
and $Conj((1,\theta)^H)\cong \mathcal{Q}(G,Fix(\theta),\theta).$
\end{lemma}

\begin{proof}
Let $g,h\in G$ and $m,k\in \mathbb{Z}$. Then
\begin{align*}
(g,\theta^k)(h,\theta^m)(g,\theta^k)^{-1}&=(g,\theta^k)(h,\theta^m)(\theta^{-k}(g)^{-1},\theta^{-k})\\
&=(g\theta^k(h),\theta^{k+m})(\theta^{-k}(g)^{-1},\theta^{-k})=(g\theta^k(h)\theta^{m}(g)^{-1},\theta^{m}).
\end{align*}
Thus if $h=1$ and $m=1$ the first statement follows. Let $\mathcal{C}=(1,\theta)^{H}$. The map
$$\mathcal{Q}(G,Fix(\theta),\theta)\longrightarrow Conj(\mathcal{C}),\quad xFix(\theta)\mapsto (x\theta(x)^{-1},\theta)$$
is a quandle isomorphism.
\end{proof}

\begin{example}\label{conjugacy class}
Let $G$ be a group, $g\in G$ and $\theta	=\widehat{g}$ be the inner automorphism with respect to $g$.  Note that $Fix(\theta)=C_G(g)$.  The map 
$$\mathcal{Q}(G,Fix(\theta),\theta)\longrightarrow Conj(G\rtimes \langle\theta\rangle)\longrightarrow Conj(g^G), \quad hC_G(g)\mapsto (h\theta(h)^{-1},\theta)\mapsto hgh^{-1},$$
is a quandle isomorphism.
\end{example}

Let us show another construction of quandles over groups. Let $G$ be a group, $\theta\in \aut{G}$ and $t\geq 1$. We define
$$\theta_t:G^t\longrightarrow G^t,\quad (x_1,\ldots,x_t)\mapsto (\theta(x_t),x_1,x_2,\ldots, x_{t-1}).$$
It is easy to check that $\theta_t$ is an automorphism of $G^t$ and that $H_t=Fix(\theta_t)=\setof{(x,x,\ldots,x)}{x\in Fix(\theta)}\cong Fix(\theta)$. We denote by $(G,t,\theta)$ the coset quandle $\mathcal{Q}(G^t,H_t,\theta_t)$. Note that, if $t=1$ then $(G,1,\theta)$ is just $\mathcal{Q}(G,Fix(\theta),\theta)$.  Moreover if $\theta$ is an inner automorphism,  then $(G,1,\theta)$ is isomorphic to a conjugation quandle over a conjugacy class by virtue of Example \ref{conjugacy class}.
%
%
%
%

\subsection{Invariant equivalence relations and congruences}
Let $Q$ be a set. We denote by $0_Q=\setof{(x,x)}{x\in Q}$ the bottom element of the lattice of the equivalence relations on $Q$ and by $1_Q=Q\times Q$ the top element. Given an equivalence relation $\alpha$ we denote by $[x]_\alpha$ (or simply by $[x]$) the class of $x\in Q$ with respect to $\alpha$ and by $Q/\alpha$ denotes the set of equivalence classes.

Let $(Q,*,\ldiv)$ be a left quasigroup. We can define two operators from the lattice of equivalence relations on $Q$ to the lattice of subgroups of the left multiplication group of $Q$. Namely:
\begin{align*}
\dis_*:\alpha\mapsto \dis_\alpha&=\langle h L_x L_y^{-1}h^{-1},\, x,\alpha\, y,\, h\in \lmlt(Q)\rangle\\
\dis^*: \alpha\mapsto\dis^\alpha&=\setof{h\in \lmlt(Q)}{h(x)\,\alpha\, x\text{ for every } x\in Q}
\end{align*}

We denote $\dis_{Q\times Q}$ just as $\dis(Q)$ and we call it {\it displacement group} of $Q$.  In particular, $\lmlt(Q)=\dis(Q)\langle L_x\rangle$ for every $x\in Q$ and so if $Q$ is idempotent they have the same orbits \cite[Lemma 1.1]{Super}.

Let $\alpha$ be an equivalence relation on $Q$. If the classes of $\alpha$ are blocks with respect to the action of $\lmlt(Q)$ we say that $\alpha$ is {\it invariant}. If $Q$ is idempotent and $x\, \alpha\, y$ then $x=L_x(x)\, \alpha \, L_x(y)$, i.e. the class of $x$ with respect to $\alpha$ is a subalgebra. It is easy to check that the set of invariant equivalence relations is a sublattice of the equivalence relations.  A left quasigroup is {\it primitive} if it has no invariant equivalence relations (i.e. the action of the left multiplication is primitive).

\begin{example}\label{ex sigma}
Let $Q$ be a left quasigroup and $N$ be a normal subgroup of $\lmlt(Q)$. Then the equivalence relation
$$x\, \sigma_N\, y\text{ if and only if } N_x=N_y$$ 
is an invariant equivalence relation.
\end{example}

\begin{example}
Let $Q$ be a quandle and $s,t$ be binary terms. We define the map 
$$p:Q\longrightarrow \mathbb{N},\quad x\mapsto |\setof{y\in Q}{t(x,y)=s(x,y)}|.$$
The relation $\sim_p$ as $x\, \sim_p\, y$ if and only if $p(x)=p(y)$ is an invariant equivalence relation. Indeed $L_z$ is a bijective between $\setof{y\in Q}{t(x,y)=s(x,y)}$ and $\setof{y\in Q}{t(L_z(x),y)=s(L_z(x),y)}$.
\end{example}

A congruence of a left quasigroup $(Q,*,\ldiv)$ is an equivalence relation $\alpha$ on $Q$ such that 
$$x*z \, \alpha \, y*t\text{  and  } x\ldiv z \, \alpha \, y\ldiv t$$ whenever $x\, \alpha \, y $ and $x\, \alpha\, t$.

Congruences and surjective homomorphisms are in one to one correspondence (see for instance \cite{UA}).  Indeed, given a congruence $\alpha$,  the operation $[x]*[y]=[x*y]$ on the factor set $Q/\alpha$ is well-defined, $(Q/\alpha,*)$ is a left quasigroup and the natural map $x\mapsto [x]$ is a morphism.  On the other hand given a surjective homomorphism $f:Q_1\longrightarrow Q_2$ then $\ker{f}=\setof{(x,y)\in Q_1\times Q_1}{f(x)=f(y)}$ is a congruence of $Q$.

Moreover, the assignment $\pi_\alpha:L_x\longrightarrow L_{[x]_\alpha}$ for every $x\in Q$ can be extended to a group homomorphism between $\lmlt(Q)$ and $\lmlt(Q/\alpha)$ with kernel denoted by $\lmlt^\alpha$. Such map $\pi_\alpha$ can be restricted and corestricted to $\dis(Q)$ and $\dis(Q/\alpha)$ and the kernel of the restricted map is $\dis^\alpha$ as defined above.

Congruences are a special class of invariant equivalence relations as shown in the following lemma.

\begin{lemma}\cite[Lemma 1.5]{semimedial}\label{simple iff}
Let $(Q,*,\ldiv)$ be a left quasigroup and $\alpha$ be an equivalence relation on $Q$. The following are equivalent:
\begin{itemize}
\item[(i)] $\alpha$ is a congruence.
\item[(ii)] $\alpha$ is invariant and $\dis_\alpha\leq \dis^\alpha$.
\end{itemize}
\end{lemma}

According to Lemma \ref{simple iff}, primitive left quasigroups are simple. The smallest example of a simple quandle that is not primitive is {\tt SmallQuandle(12,3)} from the RIG database of GAP.

\begin{corollary}
Let $Q$ be a left quasigroup and $\alpha$ be an invariant equivalence relation. If $\alpha\leq \lambda_Q$ then $\alpha$ is a congruence of $Q$.
\end{corollary}

\begin{proof}
It follows by \cite{semimedial}, since $1=\dis_\alpha\leq \dis^\alpha$.
\end{proof}

Given a left quasigroup $Q$ we can define two operators from the lattice of the subgroups of the displacement group of $Q$ to the lattice of equivalence relations of $Q$ as follows:
\begin{align*}
\mathcal{O}_*:N\mapsto \mathcal{O}_N&=\setof{(x,y)\in Q\times Q}{y=h(x) \text{ for some }h\in N},\\
\c{*}:N\mapsto \c{N}&=\setof{(x,y)\in Q\times Q}{L_x L_y^{-1}\in N}.
\end{align*}

Such relations do not need to be invariant in general. We define the lattice of {\it admissible subgroups} as $\N(Q)=\setof{N\leq \dis(Q)}{\mathcal{O}_N\leq \c{N} \text{ and } N\trianglelefteq \lmlt(Q)}$. For such subgroups the relation $\mathcal{O}_N$ is a congruence of $Q$ \cite[Corollary 1.9(ii)]{semimedial}. 

For further details on the operators $\dis_*,\dis^*,\c{*}$ and $\mathcal{O}_*$ we refer the reader to \cite{CP, semimedial, Galois}. 

\begin{lemma}\label{simple norm}
Let $(Q,*,\ldiv)$ be a simple left quasigroup. Then $\N(Q)=\{1,\dis(Q)\}$.
\end{lemma}

\begin{proof}
Assume that $N\in \N(Q)$. Then $\alpha=\mathcal{O}_N\in \{0_Q,1_Q\}$. If $\alpha=0_Q$ then $N=1$. Otherwise $\c{N}=1_Q$ and so $\dis(Q)\leq N$.
\end{proof}

Let us turn our attention to racks and quandles again. In particular, for a rack $Q$ the admissible subgroup coincide with the normal subgroups of $\lmlt(Q)$ contained in $\dis(Q)$ \cite[Lemma 3.1(ii)]{semimedial} and $\lambda_Q$ is a congruence.

\begin{lemma}\label{dis and invariant}
Let $Q$ be a rack and $\alpha$ be an invariant equivalence relation. Then: 
\begin{itemize}
\item[(i)] $\dis_\alpha,\dis^\alpha\in \N(Q)$.
\item[(ii)] $[\dis^\alpha,\lmlt(Q)]\leq \dis_\alpha$
\item[(iii)] If $Q$ is a simple quandle then $\dis_\alpha=\dis(Q)$ and $\dis^\alpha=1$.
\end{itemize}
\end{lemma}

\begin{proof}
(i) Let $h\in \lmlt(Q)$, $x\,\alpha\, y$ and $g\in \dis^\alpha$. Then
\begin{align*}
hL_x L_y h^{-1}=L_{h(x)} L_{h(y)}^{-1}\in \dis_\alpha,\quad  h g h^{-1}(x) &\,\alpha\, h h^{-1} (x)=x.
\end{align*}
Therefore $\dis_\alpha,\dis^\alpha$ are normal in $\lmlt(Q)$.

(ii) Let $h\in \dis^\alpha$ and $x\in Q$. Then $[h,L_x]=h L_x h^{-1} L_x^{-1}=L_{h(x)} L_x^{-1}\in \dis_\alpha$. We can conclude since if $g=L_{x_1}^{k_1}\ldots L_{x_n}^{k_n}$ then $[h,g]$ is a product of conjugates of $[h,L_{x_i}]$ for $i=1,\ldots,n$.
%
%
%
%
%
%

(iii) The groups $\dis_\alpha,\dis^\alpha$ are admissible and $\N(Q)=\{1,\dis(Q)\}$ according to Lemma \ref{dis and invariant}(i) and Lemma \ref{simple norm}. If $\dis_\alpha=1$ then $\alpha\leq \lambda_Q=0_Q$. If $\dis^\alpha=\dis(Q)$ then $1_Q=\mathcal{O}_{\dis(Q)}\leq \alpha$.
\end{proof}

%
%
%

Let $(Q,*,\ldiv)$ be a left quasigroup and $n\in \mathbb{Z}$,  Let us define the left quasigroup $Q_n=(Q,*_n,\ldiv_n)$ where $x*_n y=L_x^n(x)$ and $x\ldiv_n y=L_x^{-n}(y)$.

\begin{lemma}
Let $Q$ be a rack and $n\in \mathbb{Z}$. Then
\begin{itemize}
\item[(i)] $\lambda_{Q_n}$ is an invariant relation of $Q$. 
\item[(ii)] $\dis(Q_n)=\langle L_x L_y^{-1},\, x,y\in Q\rangle \in \N(Q)$.
\end{itemize}
\end{lemma}

\begin{proof}
Both the statements follow since $L_{h(x)}^n=h L_x^n h^{-1}$ for every $h\in \lmlt(Q)$.
\end{proof}

\section{Conjugation quandles}\label{Sconj}
 
\subsection{Group theoretical results}
Let $G$ be a group, $x\in G$. We say that $x^G$ has the $\star$-property if the following holds:
\begin{equation}\label{superconnected conj}
	\forall y\in x^G \, \text{ exists } h\in \langle x,y\rangle \, \text{such that}\,  y=hxh^{-1}.\tag{$\star$} \end{equation}
\begin{lemma}\label{star implies N=C}
Let $G$ be a group and assume that $x^G$ has $\star$-property. Then $C_G(y)=N_G(\langle y\rangle)$ and $x^G\cap C_G(y)=\{y\}$ for every $y\in x^G$.
\end{lemma}
\begin{proof}
It is enough to prove the statement for $x$, since the normalizers and the centralizers of elements of $x^G$ are conjugate. If $y\in x^G\cap C_G(x)$ then $y$ and $x$ and conjugate in the abelian group $\langle x,y\rangle$. Therefore $y=x$.

If $y\in N_G(x)$ then $z=yxy^{-1}\in C_G(x)\cap x^G$ and so $z=x$, i.e. $y\in C_G(x)$.
\end{proof}

Let us introduce some notation. Let $G$ be a group and $\pi$ be a set of prime numbers;
\begin{itemize}
\item[(i)] $x\in G$ is a $\pi$-element if the primes dividing $|x|$ are those in $\pi$.
\item[(ii)] A subgroup $H\leq G$ has a normal $\pi$-complement, if there exists a normal subgroup $K$ of $G$ such that $G=KH$ and $|K|$ is coprime with the elements of $\pi$.
\item[(iii)] $O_{\pi'}(G)$ is the largest normal subgroup of $G$ with order coprime with the elements of $\pi$.
\end{itemize}

%
%


The following theorem shows a sufficient condition for a conjugacy class to satisfy the $\star$-property.

\begin{theorem} \label{pi'2}\cite[Theorem 1]{Robinson}
Let $G$ be a group, $x\in G$ and $x$ be a $\pi$-element of $G$. The following are equivalent:
\begin{itemize}
\item[(i)] $G=O_{\pi'}(G)C_G(x)$.
\item[(ii)] $\langle x, x^g\rangle$ has a normal $\pi$-complement and $[O_{\pi'}(\langle x,x^g\rangle),x]\cap C_G(x)\leq O_{\pi'}(C_G(x))$ for all $g\in G\setminus C_G(x)$.
\end{itemize}
If one of the previous properties hold then $x^G$ has the $\star$-property.
\end{theorem}

The last statement of Theorem \ref{pi'2} is not contained in the original claim in \cite{Robinson}. Nonetheless, as a step of the proof of the implication (ii) $\Rightarrow$ (i) (at the beginning of page 716) the statement is verified to be true.

\begin{remark}\label{rem on complement}
Let $G$ be a group and $x\in G$ be a $\pi$-element. If $\langle x\rangle$ has a normal $\pi$-complement $N$ then $N\leq O_{\pi'}(G)$ and $x\in C_G(x)$.  Thus $G=N\langle x\rangle=O_{\pi'}(G)C_G(x)$.
\end{remark}

%
%
%


Let us consider the case of conjugacy classes of elements of prime power order.

\begin{theorem}\label{O_p'} 
Let $G$ be a group, $p$ a prime and $x\in G$ be a $p$-element of $G$.  The following are equivalent:
\begin{itemize}
\item[(i)] $x^G\cap C_G(x)=\{x\}$.

\item[(ii)] $G=O_{p'}(G)C_G(x)$.

\item[(iii)] $x^G$ has the $\star$-property.

\end{itemize}

\begin{proof}
(i) $\Rightarrow$ (ii)   Let $P$ be a $p$-Sylow subgroup containing $x$. Let us show that $x\in Z(P)$. Consider the conjugate of $x$ in $P$, i.e. $C=\setof{x^y}{y\in P}$. The element $x$ acts on $C$ by conjugation and since $x^G\cap C_G(x)=\{x\}$ then it fixes just $x$. Therefore $|C|=1\pmod p$. But this implies that $x\in Z(P)$. Therefore $x^G\cap P\leq x^G\cap C_G(x)=\{x\}$.  Therefore we can apply Theorem 3.1 of \cite{GGLN} and we get that $G=O_{p'}(G)C_G(x)$.

(ii) $\Rightarrow$ (iii) It follows by Theorem \ref{pi'2}.

(iii) $\Rightarrow$ (i) According to Lemma \ref{star implies N=C} then $x^G\cap C_G(x)=\{x\}$.
\end{proof}
%

%
\end{theorem}
%
%

%
\begin{proposition}\label{star prop and complement}
Let $G$ be a group and let $\langle x\rangle $ be a $p$-Sylow subgroup of $G$. Then the following are equivalent:
\begin{itemize}
\item[(i)] $x^G$ has the $\star$-property.
\item[(ii)] $\langle x\rangle$ has a normal $p$-complement,  i.e. $G\cong H\rtimes \langle x\rangle$ for some normal subgroup $H$ such that $|H|$ and $p$ are coprime.
\end{itemize}
\end{proposition}

\begin{proof}
(i) $\Rightarrow$ (ii) According to Lemma \ref{star implies N=C} $C_G(x)=N_G(\langle x\rangle)$. Therefore by the Burnside $p$-complement theorem $\langle x\rangle$ has a normal complement \cite[Theorem II, section 243]{Burn}. 

(ii) $\Rightarrow$ (i) Let $N$ be the normal complement of $\langle x\rangle$ and let $y\in x^G$. Then $\langle y\rangle$ is also a $p$-Sylow subgroup of $G$. In particular $\langle x\rangle$ and $\langle y\rangle$ are $p$-Sylow subgroups of $H=\langle x,y\rangle$ and so they are conjugate by an element $h\in H$. In particular $hxh^{-1}\in y^G\cap C_G(y)=\{y\}$, therefore $x$ and $y$ are conjugate in $H$.
\end{proof}

%

\subsection{Superfaithful and superconnected conjugation quandles}

Let $G$ be a group and $H$ be a subset of $G$ closed under conjugation. As mentioned earlier the set $H$ endowed with the conjugation operation is a quandle denoted by $Conj(H)$. The main examples of such quandles are conjugacy classes in groups.  If $x\in G$ we simply denote the conjugation quandle over the conjugacy class of $x$ in $G$ by $x^G$ rather than by $Conj(x^G)$.

Let $\mathcal{C}\subseteq G$ be a conjugation quandle and $H=\langle \mathcal{C}\rangle$. Then $L_x=\widehat{x}|_{\mathcal{C}}$ for every $x\in \mathcal{C}$, where $\widehat{x}$ denote the inner automorphism of $G$ with respect to $x$.

A quandle $Q$ is superfaithful if and only if $Fix(L_x)=\{x\}$ for every $x\in Q$ \cite[Lemma 2.3]{Super}. Let us study such condition for conjugation quandles.

\begin{lemma}\label{super then C=N}
Let $\mathcal{C}$ be a conjugation quandle and $H=\langle \mathcal{C}\rangle$.  If $\mathcal{C}$ is superfaithful then $N_H(\langle x\rangle)=C_H(x)$ is self-normalizing and $\mathcal{C}\cap N_H(\langle x\rangle)=\{x\}$ for every $x\in \mathcal{C}$.
\end{lemma}

\begin{proof}
The quandle $\mathcal{C}$ is superfaithful, so if $y* x=yxy^{-1}=x$ then $y=x$.  In particular $\mathcal{C}\cap C_H(x)=\{x\}$.  Assume that $g\in N_H(\langle x\rangle)$, i.e.  $g* x=gxg^{-1}=x^n$ for some $n\in\mathbb{N}$.  Then $x^n* x=x$, i.e. $x^n=x$ and so $g\in C_G(x)$.  Let $g\in N_H(C_H(x))$.  Then $x^g\in \mathcal{C}\cap C_H(\langle x\rangle)=\{x\}$, and so $g\in C_H(x)$.
%
%
\end{proof}

\begin{lemma}
Let $G$ be a group and $x\in G$. If $x^G\cap N_G(\langle x\rangle)=\{x\}$ then $x^G$ is superfaithful.
\end{lemma}
\begin{proof}
Let $y*x=x$ for $y\in \mathcal{C}$. Then there exists $g\in G$ such that $y=x^g\in x^G\cap C_G(x)=\{x\}$.  Therefore $x^g=x$,  i.e.  $Q$ is superfaithful.
\end{proof}
%

Let us introduce a criterion to check if a quandle is superconnected.
\begin{proposition}\label{xy conjugated}
Let $\mathcal{C}$ be a conjugation quandle and $H=\langle \mathcal{C}\rangle$.  The following are equivalent:
\begin{itemize}
\item[(i)] $\mathcal{C}$ is superconnected 
\item[(ii)] $\mathcal{C}=x^H$ and has the $\star$-property.
\end{itemize}
\end{proposition}

\begin{proof}
A quandle is superconnected if and only if every $2$-generated subquandle is connected \cite{Super}. The subquandle $S$ generated by $x,y\in \mathcal{C}$ is connected if and only if $x$ and $y$ are in the same orbit with respect to the action of $\lmlt(S)$ \cite[Corollary 5.5]{semimedial}. According to \cite[Lemma 5.3]{semimedial} $\lmlt(S)$ is generated by $L_x=\widehat{x}|_S$ and $L_y=\widehat{y}|_S$. Therefore $S$ is connected if and only if there exists $h\in \langle x,y\rangle$ such that $y=hxh^{-1}$.

In particular $\mathcal{C}$ is connected and so $\mathcal{C}=x^H$ for some $x\in \mathcal{C}$.
\end{proof}


We conclude this subsection applying the group theoretical results developed at the beginning of this section to display some examples of superconnected quandle.

\begin{proposition}\label{semidirect}
Let $Q$ be a quandle,  $x\in Q$ such that $|L_x|$ and $|\dis(Q)|$ are coprime. Then: 

\begin{itemize}
\item[(i)] the orbit of $L_x$ in $Q/\lambda_Q$ is superconnected.
\item[(ii)] If $Q$ is connected then $Q/\lambda_Q$ is superconnected. 
\end{itemize}
\end{proposition}

\begin{proof}
(i) According to Proposition \ref{xy conjugated} we just need to show that $L_x^{\lmlt(Q)}$ has the $*$-property. The orbit of $L_x$ in $Q/\lambda_Q$ is isomorphic to the conjugacy class of $L_x$.  The group $\dis(Q)$ is a complement to $\langle L_x\rangle$. Hence, according to Remark \ref{rem on complement} we can apply Theorem \ref{pi'2}.

(ii) It follows by (i) since the orbit of $L_x$ coincide with $Q/\lambda_Q$.
\end{proof}

\begin{corollary}\label{semidirect2}
Let $Q$ be a quandle,  $x\in Q$ such that $|L_x|$ and $|\dis(Q)|$ are coprime. Then $Q$ is superconnected if and only if $Q$ is connected and faithful. 
\end{corollary}
%

\begin{corollary}\label{U}
Let $Q$ be a quandle, $x\in Q$ and $|L_x|=n m$. If $|\dis(Q)|$ and $n$ are coprime then the orbit of $L_x^m$ in $Q_m/\lambda_{Q_m}$ is superconnected.
\end{corollary}
\begin{proof}
Since $|\dis(Q)|$ and $|L_x^m|=n$ are coprime then $n$ is also coprime with the order of $\dis(Q_m)\leq \dis(Q)$. Then we can apply Proposition \ref{semidirect} to $Q_m$.
\end{proof}

If the order of the left multiplication mappings is prime we have a complete characterization of superconnected quandles in group-theoretical terms.

\begin{proposition}\label{Prop superconnected p}
Let $Q$ be a connected quandle, $x\in Q$ and $G=\lmlt(Q)$. If $L_x$ is a $p$-element of $G$ for some prime $p$ the following are equivalent:
\begin{itemize}
\item[(i)] $Q/\lambda_Q$ is superconnected. 
\item[(ii)] $\lmlt(Q)\cong \dis(Q)\rtimes \langle L_x\rangle$ and $|\dis(Q)|$ and $|L_x|$ are coprime.

\item[(iii)] $\langle L_x\rangle$ has a normal $p$-complement.

\item[(iv)] $G=O_{p'}(G)C_G(L_x)$.
\end{itemize}

\end{proposition}

\begin{proof}
According to Theorem \ref{O_p'} and Proposition \ref{xy conjugated},  (i) and (iv) are equivalent.
 
(ii) $\Rightarrow$ (iii) is clear.

(iii) $\Rightarrow$ (iv) is clear.

(iv) $\Rightarrow$ (ii) Let $\alpha=\mathcal{O}_{O_{p'}(G)}$. Then $O_{p'}(G)\leq \lmlt^\alpha$.  Since $G=O_{p'}(G)C_G(L_x)$ then $L_{[x]}$ is central in $\lmlt(Q/\alpha)$. The quandle $Q$ is connected, so $L=\setof{L_{[y]}}{y\in Q}$ is a conjugacy class in $\lmlt(Q/\alpha)$ and so $|L|=1$. The factor $Q/\alpha$ is a quandle and thus $L_{[x]}=1$. The quandle $Q/\alpha$ is connected, so $|Q/\alpha|=1$, i.e. $\alpha=1_Q$. In particular given $x,y\in Q$ we have that $y=h(x)$ for $h\in O_{p'}(G)$ and then 
$$L_yL_x^{-1}=L_{h(x)} L_x=h L_x h^{-1} L_x^{-1}\in O_{p'}(G).$$ 
So, $\dis(Q)\leq O_{p'}(G)$ and we can conclude since $G=\dis(Q)\langle L_x\rangle$.
%
\end{proof}

\begin{corollary}
Let $Q$ be a finite superconnected quandle. If $|L_x|=2^n$ then $Q$ is solvable.
\end{corollary}

\begin{proof}
According to Proposition \ref{Prop superconnected p} $|\dis(Q)|$ is odd and so solvable. Then also $Q$ is solvable \cite[Lemma 6.2]{CP}. 
\end{proof}

%
%
%
%
%
%
%
%

\begin{theorem}
Let $Q$ be a quandle, $x\in Q$ such that $L_x$ is a $p$-element of $\lmlt(Q)$ for some prime $p$.  Then $Q$ is superconnected if and only if $Q$ is connected and superfaithful.
\end{theorem}

\begin{proof}
The quandle $Q$ is connected and superfaithful. Therefore it is isomorphic to $L_x^{\lmlt(Q)}$ and moreover $L_x^{\lmlt(Q)}\cap N_{\lmlt(Q)}(L_x)=\{L_x\}$ according to Lemma \ref{super then C=N}.  According to Theorem \ref{O_p'} then $L_x^{\lmlt(Q)}$ has the $\star$-property and so it is superconnected by Proposition \ref{xy conjugated}.
\end{proof}
%
%
%

\section{Simple Quandles}\label{simple}

\subsection{Superfaithful and superconnected simple quandles}

Simple quandles are given by quandles $(L,t,\theta)$ where $L$ is a finite simple group, see \cite{AG,JS}.

In this section we are focusing on simple superfaithful and superconnected quandles. The main tool we are going to use is the following characterization.

\begin{proposition}\label{prop on simple}\cite[Proposition 2.7]{Super}
Let $G$ be a finite group, $\theta\in \aut{G}$ and $t> 1$. The following are equivalent:
\begin{itemize}
\item[(i)] $(G,1,\theta)$ is superfaithful and $t$ is coprime with $|Fix(\theta)|$.
\item[(ii)] $(G,t,\theta)$ is superfaithful.
\end{itemize}
In particular, $(G,t,1)$ is superfaithful if and only if $|G|$ and $t$ are coprime.
\end{proposition}

Let us denote by $A_n$ the alternating group of rank $n$ and by $PSL_{n+1}(q)$ the projective special linear groups.

\begin{proposition}\label{sporadic inner}
Let $L=A_n, PSL_{n+1}(q)$ or $L$ be a sporadic group and $g\in L$. The following are equivalent:
\begin{itemize}
\item[(i)] $(L,t,\widehat{g})$ is superfaithful.
\item[(ii)] $g=1$ and $t$ is coprime with $|L|$. 
\end{itemize}
\end{proposition}

\begin{proof}

(i) $\Rightarrow$ (ii) Let $\theta=\widehat{g}$ be an inner automorphism of $L$. According to Lemma \ref{prop on simple} the quandle $Q^\prime=(L,1,\widehat{g})$ is also superfaithful and $t$ is coprime with $|C_L(g)|$.  By Example \ref{conjugacy class}, the quandle $Q^\prime$ is isomorphic to the conjugation quandle $\mathcal C=Conj(g^L)$.  If $g\neq 1,$ according to the main result in \cite{ConjugacySimple} the quandle $\mathcal{C}$ contains commuting elements, i.e. it is not superfaithful.  Thus, $g=1$ and so $t$ is coprime with $|C_G(1)|=|L|$.

(ii) $\Rightarrow$ (i) It follows by Proposition \ref{prop on simple}.
\end{proof}

A group $G$ is said to be {\it complete} if the map 
$$G\longrightarrow \aut{G}, \quad g\mapsto \widehat{g}$$
is an isomorphism, i.e. $Z(G)=1$ and every automorphism of $G$ is inner. Some finite simple groups are complete as the sporadic groups in the following list:
\begin{equation*}
\mathfrak C=\{M_{11},\, M_{23}, \,M_{24}, \,J_1,\, J_4,\, Co_1,\, Co_2,\, Co_3,\, Fi_{23}, \, He, \, Ly,\, Th, \, B,\, M\}.
\end{equation*} 

\begin{corollary}
Let $L\in \mathfrak C$ and $\theta\in \aut{L}$. The following are equivalent:
\begin{itemize}
\item[(i)] $(L,t,\theta)$ is superfaithful.
\item[(ii)] $\theta=1$ and $t$ is coprime with $|L|$. 
\end{itemize}
\end{corollary}

\begin{proof}
For such groups, all the automorphisms are inner and so we apply Proposition \ref{sporadic inner}.
\end{proof}

Let us now consider simple quandles over alternating groups.

\begin{proposition}
Let $Q=(A_n,t,\theta)$ and $n>5$.
The following are equivalent:
\begin{itemize}
\item[(i)] $Q$ is superfaithful.
\item[(ii)] $\theta=1$ and $t$ is coprime with $n!$.
\end{itemize}

\end{proposition}

\begin{proof}

(i) $\Rightarrow$ (ii) If $Q$ is superfaithful, then also $(A_n,1,\theta)$ is superfaithful by Proposition \ref{prop on simple} and $t$ is coprime with $|Fix(\theta)|$. According to Lemma \ref{quandles as conj}, we have, $$(A_n,1,\theta)=\mathcal{Q}(A_n,Fix(\theta),\theta)\cong Conj(\setof{(x\theta(x)^{-1},\theta)}{x\in A_n})\subseteq Conj(A_n\rtimes \langle \theta\rangle).$$

If $\theta$ is an inner automorphism we can conclude by Proposition \ref{sporadic inner}. Let assume that $\theta$ is not inner, we have
\begin{align*}
[(x\theta(x)^{-1},\theta),(1,\theta)]&=(1,\theta)(x\theta(x)^{-1},\theta)(1,\theta)^{-1}(x\theta(x)^{-1},\theta)^{-1}\\
&=(\theta(x)\theta^2(x)^{-1},\theta)(x\theta^{-1}(x)^{-1},\theta^{-1})\\
&=(\theta(x)\theta^2(x)^{-1}\theta(x)x^{-1},1)
\end{align*} 
for every $x\in A_n$. If $n\neq 6$, we can assume that $\theta$ is the inner automorphism with respect to an element of $\sym{n}$ restricted to $A_n$, say $\tau$. Thus we have
\begin{align*}
\theta(x)\theta^2(x)^{-1}\theta(x)x^{-1}=\tau x\tau^{-1}\tau^2 x^{-1} \tau^{-2}\tau x \tau^{-1}x^{-1}=\tau x\tau x^{-1} \tau^{-1} x \tau^{-1}x^{-1}=[\tau,x\tau x^{-1}].
\end{align*}
Since $\sym{n}=\langle A_n,\tau\rangle$ we have that $\tau^{\sym{n}}=\tau^{A_n}$. Every non-trivial conjugacy class of $\sym{n}$ contains commuting elements, so we can choose $x\in A_n$, such that $[\tau, x \tau x^{-1}]=1$ and $x\tau x^{-1}\neq \tau$, i.e. $x\notin Fix(\theta)$. Thus, $\tau=1$ and accordingly $t$ is coprime with $|A_n|=\frac{n!}{2}$ and so it is coprime with $n!$ (since $n>5$). If $n=6$ a GAP computer search reveals that $(A_6,1,\theta)$ is not superfaithful whenever $\theta\neq 1$.

(ii) $\Rightarrow$ (i) It follows by Proposition \ref{sporadic inner}.
\end{proof}


Let $L$ be a finite simple group, $\theta\in \aut{L}$ and $t\in \mathbb{N}$.  The simple quandle $Q=(L,t,\theta)$ can be identified with the conjugacy class $\mathcal{C}$ of $(1,\ldots,1,\theta)\in L^t\rtimes \langle \theta\rangle$ (see Example \ref{conjugacy class}).  Then $Q$ is superconnected if and only if every pair of elements $x,y\in \mathcal{C}$ is conjugated in the sugroup $\langle x,y\rangle \leq L^t\rtimes \langle \theta\rangle$.
\begin{corollary}\label{(L,t,1) super}
Let $Q=(L,t,1)$. The following are equivalent:
\begin{itemize}
\item[(i)] $Q$ is superfaithful.
\item[(ii)] $Q$ is superconnected.
\item[(iii)] $|L|$ and $t$ are coprime.
\end{itemize}
\end{corollary}

\begin{proof}
We just need to prove that (iii) $\Rightarrow$ (ii). The displacement group is isomorphic to $L^t$ and the order of $L_x$ is $t$. Hence we can apply Corollary \ref{semidirect2}.
\end{proof}

In particular, Corollary \ref{(L,t,1) super} shows that Conjecture \ref{conj1} is not true.

Combining the results of the previous section and Corollary \ref{(L,t,1) super} we can make a (non-exhaustive) list of superconnected simple groups. 

\begin{theorem}
Let $L$ be a simple group then:
\begin{itemize}
\item[(i)] if $n>5$ and $L=A_n$ then $(L,t,\theta)$ is superconnected if and only if $\theta=1$ and $t$ is coprime with $n!$.
\item[(ii)] If $L=A_n(q)$ or $L$ is a sporadic group, and $g\in L$ then $(L,t,\widehat{g})$ is superconnected if and only if $g=1$ and $t$ is coprime with $|L|$.
\item[(iii)] If $L$ is a sporadic group in $\mathfrak{C}$ then $(L,t,\theta)$ is superconnected if and only if $\theta=1$ and $t$ is coprime with $|L|$.
\end{itemize}\end{theorem}


\subsection{Primitive quandles}\label{primitive}

A quandle $Q$ is primitive if and only if $\lmlt(Q)$ is primitive, i.e. $\lmlt(Q)_x$ is a maximal subgroup in $\lmlt(Q)$.

\begin{lemma}\label{maximality}
Let $Q$ be a quandle and $x\in Q$. Then: 
\begin{itemize}
\item[(i)]$\lmlt(Q)_x=\dis(Q)_x\langle L_x\rangle$. 
\item[(ii)] $Q$ is primitive if and only if $\dis(Q)_x$ is maximal $\widehat{L_x}$-invariant subgroup in $\dis(Q)$.
\end{itemize}

\end{lemma}

\begin{proof}
(i) Let $h\in \lmlt(Q)_x$. Then $h=h_1 L_x^n$ for some $h_1\in \dis(Q)$. Thus $h(x)=h_1(x)=x$, i.e. $h_1\in \dis(Q)_x$.

(ii) Let $\dis(Q)_x$ be a $\widehat{L_x}$-invariant maximal in $\dis(Q)$ and $h=h_1 L_x^n\in \lmlt(Q)$. Then 
$$K=\langle h,\lmlt(Q)_x\rangle=\langle h_1,\, \dis(Q)_x,\, L_x\rangle=\langle L_x^k h_1 L_x^{-k},\,k\in \mathbb{Z},\, \dis(Q)_x,L_x\rangle.$$ Then $N=\langle L_x^k h_1 L_x^{-k},\, k\in \mathbb{Z},\,\dis(Q)_x	\rangle$ is a $\widehat{L_x}$-invariant subgroup of $\dis(Q)$ containing $\dis(Q)_x$. Thus $N=\dis(Q)$ and therefore $K=\lmlt(Q)$. So, $\lmlt(Q)_x$ is maximal.
%

Assume that $\lmlt(Q)_x$ is maximal.  Let $\dis(Q)_x\leq N\leq \dis(Q)$ such that $L_x N L_x^{-1}=N$.  Then $\langle N, \lmlt(Q)_x\rangle=\langle N,\dis(Q)_x,L_x\rangle=\langle N,  L_x\rangle=N\langle L_x\rangle=\lmlt(Q)$.  Thus, $N$ is normal in $\lmlt(Q)$ contained in $\dis(Q)$, and so $N=\dis(Q)$ since $Q$ is simple.
%
\end{proof}
%

Primitive quandles belong to two families as we prove in the next lemma.

\begin{lemma}
Let $Q=(L,t,\theta)$ be a primitive quandle. Then either $t=1$ or $\theta=1$.
\end{lemma}

\begin{proof}
The stabilizer of $Fix(\theta_t)\in Q$ is $H=\setof{(x,\ldots,x,\theta^n)}{x\in Fix(\theta),\, n\in \mathbb{Z}}$ and it is maximal in $L^t\rtimes \langle \theta\rangle$. The subgroup $H$ is contained in the subgroup $K=\setof{(x_1,\ldots,x_t,\theta^n)}{x_i\in Fix(\theta),\, n\in \mathbb{Z}}$. If $H=K$ then $t=1$, otherwise if $K=L^t\rtimes \langle\theta\rangle$ then $\theta=1$.
\end{proof}

We can completely characterize the primitive quandles among the simple quandles of the form $(L,t,1)$.

\begin{lemma}\label{lemma subgrouop containing diagonal}
Let $L$ be a simple group, $t\in \mathbb{N}$ and $D=\setof{(x,\ldots,x)}{x\in L}\leq L^t$.  If $D\leq H\leq L^t$ then there exists an equivalence relation $\alpha$ on $\{1,\ldots,t\}$ such that
$$H=\setof{(x_1,\ldots,x_t)}{i\, \alpha\, j\,\Rightarrow\, x_i=x_j}.$$
\end{lemma}

\begin{proof}
Let $x\in L^t$. We define $supp(x)=\setof{i\in \{1,\ldots, t\}}{x_i\neq 1}$. Consider the subsets $\Delta\subseteq \{1,\ldots,t\}$ such that $supp(y)\subseteq \Delta$ for some $y\in H$. Let $\Delta$ be a minimal subset with such property and $i,j\in \Delta$. Let assume that there exists $x\in H$ such that $x_i\neq x_j$. Then
$$z=(x_i^{-1},\ldots,x_i^{-1})(x_1,\ldots ,x_i,\ldots ,x_j,\ldots)=(x_i^{-1} x_1,\ldots ,1,\ldots ,x_i^{-1}x_j,\ldots)\in H.$$ 
Let $y\notin C_L(x_j)$ and $u=x_i^{-1}x_j\neq 1$. The set $\setof{u^g}{g\in L}$ generates $L$ and so $y=u^{g_{s_1}}\ldots u^{g_{s_t}}$. Thus, 
\begin{align}\label{conj}
h=z^{(g_{s_1},\ldots ,g_{s_1})}\ldots z^{(g_{s_t},\ldots g_{s_t})}=(\ldots ,1,\ldots,y,\ldots)\in H.
\end{align}
Then $supp([x,h])\subseteq supp(x)\cap supp(h)\subset \Delta\setminus \{i\}$. The subset $\Delta$ is minimal and so $x_i=x_j$. 

Consider the set $\Phi$ of the minimal subsets $\Delta$ with the property above. Define the equivalence relation $\alpha$ of $\{1,\ldots,t\}$ as the smallest equivalence relation containing the relation: 
\begin{align*}
i\,\sim \, j \text{ if and only if } i,j\in \Delta \text{ for some } \Delta\in \Phi.
\end{align*}

Then $H\leq S=\setof{(x_1,\ldots,x_t)}{i\, \alpha\, j\,\Rightarrow\, x_i=x_j}$.

Let us show that $S\leq H$. Let $\Delta\in \Phi$ and $supp(x)\subseteq \Delta$ for $x\in S$. Since $x_i=x_j$ for every $i,j\in \Delta$ then $supp(x)=\Delta$.  If $h\in L$ and $x_i=g$ for $i\in \Delta$ then $h$ is the product of conjugate of $g$ and so, as in \eqref{conj} we get that there exists an element $y\in H$ such that $y_i=h$ for every $i\in \Delta$.

Let $z\in S$. Let us proceed by induction on the size of $supp(z)$. If $|supp(z)|=0$ we are done. Otherwise,  let $x\in H$ such that $supp(x)=\Delta$ and $x_i=z_i$ for every $i\in \Delta$. Then $z=xz'$ such that $z'_i=1$ for every $i\in \Delta$ and $z'_j=z_j$ otherwise. In particular, $z'\in S$ and so by induction $z'\in H$. Hence we can conclude that also $z\in H$.
\end{proof}

\begin{theorem}
Let $Q=(L,t,1)$ be a quandle.  Then $Q$ is primitive if and only if $t$ is prime.
\end{theorem}

\begin{proof}
For $Q=(L,t,1)$ and $D=\setof{(x,\ldots,x)}{x\in L}\leq L^t$ we have that $\dis(Q)\cong L^t$ and $D$ is the stabilizer of the element $Fix(1_t)\in Q$.  

Let $t=mn$. Then $\lambda_{Q_n}$ is a proper invariant equivalence relation of $Q$. Indeed $L_{xFix(1_t)}^n=L_{Fix(1_t)}^n=1_t^n$ if and only if $$x=(\underbrace{y,y,\ldots,y}_m)$$  where $y=(z_1,z_2,\ldots,z_n)$ for $z_i\in L$ and $i=1,\ldots,n$. 

Assume that $t$ is prime and let $H$ be a $\widehat{L_x}$-invariant subgroup of $L^t$ containing the diagonal $D$.  By Lemma \ref{lemma subgrouop containing diagonal} there exists an equivalence relation $\alpha$ such that
$$H=\setof{(x_1,\ldots,x_t)}{i\, \alpha\, j\,\Rightarrow\, x_i=x_j}.$$
Assume that $i\, \alpha\, j$. Since $H$ is $\widehat{L_x}$-invariant then it is easy to check that if $(x_1,\ldots,x_t)\in H$ and $x_i=x_j$ then $x_{i-1}=x_{j-1}$. Therefore then either $H=D$ or $H=L^t$. Therefore $D$ is a maximal $\widehat{L_x}$-invariant subgroup. Thus $Q$ is primitive by Lemma \ref{maximality}. 
\end{proof}


Let $n\in \mathbb{N}$. We define $D(n)=\setof{p\in \mathbb{N}}{p \text{ is prime such that }p\,|\,n}$ and $rad(n)=\prod_{p\in D(n)} p$.

\begin{proposition}\label{on rad}
Let $Q=(L,t,\theta)$ be a primitive quandle. If $rad(|L_x|)$ does not divide $rad(|L|)$ then $Q$ is superconnected.
\end{proposition}

\begin{proof}
Since $rad(|L_x|)$ does not divide $rad(|L|)$, there exist $p$ dividing $|L_x|$ but not dividing $|L|$. Thus, $|L_x|=p^n m$ and $|\dis(Q)|=|L|^t$ is coprime with $p$. Consider the quandle $Q_m$. Since $\dis(Q_m)\in \N(Q)$, then $\dis(Q_m)=\dis(Q)$ and therefore $Q_m$ is connected. Moreover $\lambda_{Q_m}$ is an invariant equivalence relation for $Q$ and therefore $Q_m$ is faithful. So, by Corollary \ref{U}, $Q_m$ is supeconnected.  Accordingly, for every $x,y\in Q$ there exists $h\in \langle L_x^m,L_y^m\rangle\leq \langle L_x,L_y\rangle$ such that $y=h(x)$. Hence also $Q$ is superconnected.
%
%
\end{proof}

Proposition \ref{on rad} can be applied only to non-sporadic groups and non-inner automorphisms. Indeed, let $Q=(L,t,\theta)$:
\begin{itemize}
\item[(i)] if $\theta=\widehat{g}$ then $rad(|L_x|)=rad(|\widehat{g}|)$ divides $rad(|L|)$.

\item[(ii)] If $L$ is a sporadic group then $\theta^2$ is inner ($Out(L)$ has size $2$ \cite{Autspor}) and so $|\theta|=gcd(2,|\theta|)|\theta^2|$.  Therefore $rad(|\theta|)$ divides $rad(|L|)$. 
\end{itemize}

\bibliographystyle{amsalpha}
\bibliography{references} 
\newpage

\end{document}